\documentclass[12pt]{amsart}
\usepackage[headings]{fullpage}
\usepackage{amssymb,epic,eepic,epsfig,amsbsy,amsmath,amscd}
\usepackage[all]{xy}
\numberwithin{equation}{section}
                        \theoremstyle{plain}
\usepackage{mathrsfs}

\newcommand{\psdraw}[2]
         {\begin{array}{c} \hspace{-1.3mm}
         \raisebox{-4pt}{\psfig{figure=#1.eps,width=#2}}
         \hspace{-1.9mm}\end{array}}
         
\newcommand\no[1]{}

\newtheorem{theorem}{Theorem}[section]
\newtheorem{thm}{Theorem}
\newtheorem{lemma}[theorem]{Lemma}

\newtheorem{proposition}[theorem]{Proposition}

\theoremstyle{definition}
\newtheorem{remark}[theorem]{Remark}

\def\BC{\mathbb C}

\def\BZ{\mathbb Z}
\def\BR{\mathbb R}
\def\BT{\mathbb T}
\def\BQ{\mathbb Q}

\def\la{\langle}
\def\ra{\rangle}

\DeclareMathOperator{\tr}{\mathrm tr}

\def\be { \begin{equation} }
\def\ee { \end{equation} }

\begin{document}

\title[Left-orderable fundamental groups and Dehn surgeries]{On Left-orderable fundamental groups and Dehn surgeries on knots}

\author{Anh T. Tran}
\address{Department of Mathematics, The Ohio State University, Columbus, OH 43210, USA}
\email{tran.350@osu.edu}


\begin{abstract}
We show that the resulting manifold by $r$-surgery on a large class of two-bridge knots has left-orderable fundamental group if the slope $r$ satisfies certain conditions. This result gives a supporting evidence to a conjecture of Boyer, Gordon and Watson that relates $L$-spaces and the left-orderability of their fundamental groups.  
\end{abstract}

\maketitle

\section*{Introduction}

The motivation of this paper is a conjecture of Boyer, Gordon and Watson that relates $L$-spaces and the left-orderability of their fundamental groups. Let $Y$ be a closed, connected, oriented 3-manifold, and denote by $\widehat{HF}(Y)$ the `hat' version of Heegaard Floer homology of $Y$. We are interested in a class of manifolds with minimal Heegaard Floer homology which was introduced in \cite{OS}. A rational homology sphere $Y$ is called an $L$-space if $\widehat{HF}(Y)$ is a free abelian group whose rank coincides with the number of elements in $H_1(Y;\BZ)$. Examples of $L$-spaces include lens spaces as well as all spaces with elliptic geometry \cite{OS}. It is natural to ask if there are characterizations of $L$-spaces which do not refer to Heegaard Floer homology. 

A non-trivial group $G$ is called left-orderable if there exists a strict total ordering $<$ on its elements such that $g<h$ implies $fg<fh$ for all elements $f,g,h \in G$. It is known that the fundamental group of an irreducible $3$-manifold with positive first Betti number is left-orderable \cite{HSt, BRW}. There is a conjectured connection between $L$-spaces and the left-orderability of their fundamental groups. Precisely, a conjecture of Boyer, Gordon and Watson \cite{BGW} states that an irreducible rational homology $3$-sphere is an $L$-space if and only if its fundamental group is not left-orderable. The conjecture was confirmed for Seifert fibered manifolds, Sol manifolds, double branched covers of non-splitting alternating links \cite{BGW}. 

In a related direction, it was shown that if $-4 \le r \le 4$ then $r$-surgery on the figure-eight knot yields a manifold whose fundamental group is left-orderable \cite{BGW, CLW}. Recently, Hakamata and Teragaito have generalized this result to all hyperbolic twist knots. They show that if $0 \le r \le 4$ then $r$-surgery on any hyperbolic twist knot yields a manifold whose fundamental group is left-orderable \cite{HT1, HT2}. In this paper, we study the left-orderability of the fundamental group of manifolds obtained by Dehn surgeries on a large class of two-bridge knots that includes all twist knots. Let $J(k,l)$ be the knot in Figure 1. 
Note that $J(k,l)$ is a knot if and only if $kl$ is even, and is the trivial knot if $kl=0$.  
Furthermore, $J(k,l)\cong J(l,k)$ and 
$J(-k,-l)$ is the mirror image of $J(k,l)$. Hence, without loss of generality, we consider $J(k,2n)$ for $k>0$ and $|n|>0$ only. When $k=2$, $J(2,2n)$ presents the twist knot. Note that the twist knot $K_n$  in \cite{HT2} is $J(-2,2n)$, which is the mirror image of $J(2,-2n)$.


\begin{figure}[htpb]
$$ \psdraw{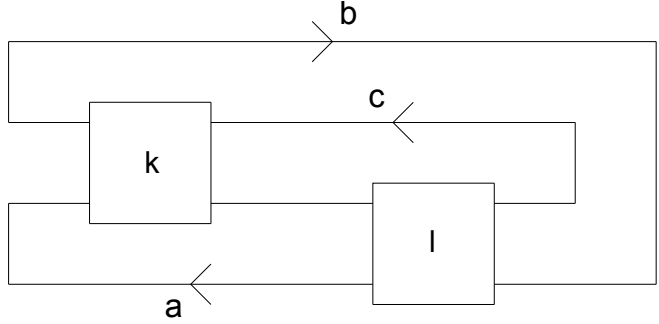}{4.25in} $$
\caption{The knot $K=J(k,l)$. Here $k$ and $l$ denote 
the numbers of half twists in each box. Positive numbers correspond 
to right-handed twists and negative numbers correspond to left-handed 
twists respectively.}
\end{figure}

The main result of the paper is as follows. 

\begin{thm}
Let $m$ and $n$ be integers such that $m \ge 1$. Suppose $r \in \BQ$ satisfies  $$ r \in
\begin{cases}
(-\max \{ 4m,4n \},0], & n \ge 2 \text{~and~} m \ge 2,\\
\left(-(4n+2),-(\frac{4(2n-1)}{\omega_n}+4), \right) \cup [-4,0],& n \ge 2 \text{~and~} m=1,\\ 
\left(-(4m+2),-(\frac{4(2m-1)}{\omega_m}+4), \right) \cup [-4,0],& n=1 \text{~and~} m \ge 2,\\ 
(-4m,-4n),& n \le -1.\\
\end{cases}
$$
where $\omega_m$ (resp. $\omega_n$) is the unique real solution of the equation $te^t=4(2m-1)$ (resp. $te^t=4(2n-1)$). Then the resulting manifold by $r$-surgery on the hyperbolic knot $J(2m,2n)$ has left-orderable fundamental group.
\label{main}
\end{thm}

\begin{remark} a) It is known that $J(k,l)$ is a hyperbolic knot if and only if $|k|, |l| \ge 2$ and $J(k,l)$ is not the trefoil knot. We exclude $J(2,2)$ from Theorem \ref{main} since it is the trefoil knot.

b) Since $J(-2m,-2n)$ is the mirror image of $J(2m,2n)$, the following follows from Theorem 1. Let $m$ and $n$ be integers such that $m \ge 1$. Suppose $r \in \BQ$ satisfies  $$ r \in
\begin{cases}
[0,\max \{ 4m,4n \}), & n \ge 2 \text{~and~} m \ge 2,\\
[0,4] \cup \left(\frac{4(2n-1)}{\omega_n}+4,4n+2\right),& n \ge 2 \text{~and~} m=1,\\ 
[0,4] \cup \left(\frac{4(2m-1)}{\omega_m}+4,4m+2\right),& n=1 \text{~and~} m \ge 2,\\ 
(4n,4m),& n \le -1.\\
\end{cases}
$$
Then the resulting manifold by $r$-surgery on the hyperbolic knot $J(-2m,-2n)$ has left-orderable fundamental group.

c) Since $J(2m,2n)$ does not yield an $L$-space by any non-trivial Dehn surgery \cite{OS}, Theorem \ref{main} gives a supporting evidence to the conjecture of Boyer, Gordon and Watson.
\end{remark}

\subsection*{Plan of the paper} In Sections 1, 2 and 3, we respectively study the knot group, the non-abelian $SL_2(\BC)$-representation space and the canonical longitude of the knot $J(2m,2n)$. Sections 4 and 5 contain crucial calculations involving the meridian and the canonical longitude of $J(2m,2n)$ which will be needed in the proof of the main theorem in the last section. Section 6 is devoted to the proof of Theorem \ref{main}.

\subsection*{Acknowledgement} We would like to thank M. Teragaito for helpful discussions and correspondence, especially for pointing out the fact that the knot has genus one is crucial in the proof of the main results in \cite{HT2} and this paper. R. Hakamata and M. Teragaito have independently obtained a similar result to Theorem \ref{main} \cite{HT3}. We would also like to thank the referee for helpful suggestions/comments.

\section{Knot groups}
\label{G_K}

Let $X$ be the closure of $S^3$ minus a tubular neighborhood of a knot $K$. The fundamental group of $X$ is called the knot group of $K$ and is denoted by $\pi_1(K)$. By \cite[Section 4]{HSn}, the knot group of $K=J(2m,2n)$ has a presentation $$\pi_1(K) = \la a,b~|~aw^n=w^nb \ra,$$ 
where $w =(ab^{-1})^m(a^{-1}b)^m$ and $a,b$ are meridians of $K$ depicted in Figure 1. 

In the case of $m=1$ (twist knots), the following presentation is more useful. Let $c$ be the meridian of $J(2,2n)$ depicted in Figure 1.

\begin{lemma}
One has $$\pi_1(J(2,2n))=\la b,c \mid bu=uc \ra$$ where $u=(b^{-1}c)^{n}c(b^{-1}c)^{-n}$.
\label{knot-group}
\end{lemma}

\begin{proof}
Let $b_1, \cdots, b_{|n|+1}$ and $c_1, \cdots, c_{|n|+1}$ be meridians of $K=J(2,2n)$ depicted in Figures 2 and 3, where $b_1=b$ and $c_1=c$. 

\textit{\underline{Case 1}: $n <0$.} From the Wirtinger relations corresponding to the bottom $2|n|$ (positive) crossings of $K$, it follows that $b_{j+1}=c_j^{-1}b_jc_j$ and $c_{j+1}=b_{j+1}c_jb_{j+1}^{-1}$. Then, by induction on $j$, we have $b_{j+1}=(c^{-1}b)^{j}b(c^{-1}b)^{-j}$ and $c_{j+1}=(c^{-1}b)^{j}c(c^{-1}b)^{-j}$.

\begin{figure}[htpb]
$$ \psdraw{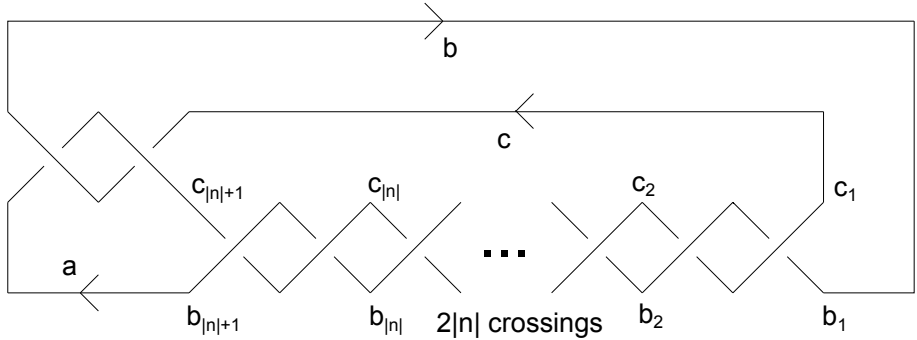}{4.75in} $$
\caption{$J(2,2n),~n<0$.}
\end{figure}

\textit{\underline{Case 2}: $n >0$.} From the Wirtinger relations corresponding to the bottom $2|n|$ (negative) crossings of $K$, it follows that $c_{j+1}=b_j^{-1}c_jb_j$ and $b_{j+1}=c_{j+1}b_jc_{j+1}^{-1}$. Then, by induction on $j$, we have $c_{j+1}=(b^{-1}c)^{j}c(b^{-1}c)^{-j}$ and $b_{j+1}=(b^{-1}c)^{j}b(b^{-1}c)^{-j}$.

\begin{figure}[htpb]
$$ \psdraw{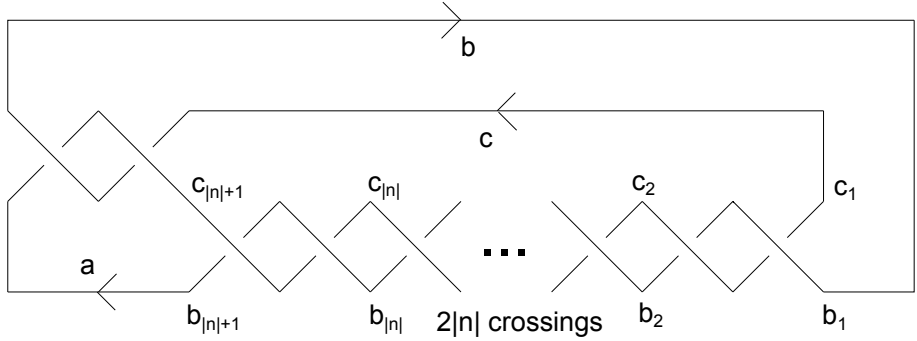}{4.75in} $$
\caption{$J(2,2n),~n>0$.}
\end{figure}

In both cases, we have $b_{|n|+1}=(b^{-1}c)^{n}b(b^{-1}c)^{-n}$ and $c_{|n|+1}=(b^{-1}c)^{n}c(b^{-1}c)^{-n}$. The Wirtinger relations corresponding to the top 2 (negative) crossings of $K$ are equivalent to the same relation $c=c_{|n|+1}^{-1}bc_{|n|+1}$. The lemma follows by letting $u=c_{|n|+1}$.
\end{proof}

\begin{remark}
The above presentation of the knot group of $J(2,2n)$ follows from the choice of generators of its Kauffman bracket skein algebra in \cite{GN} and is very useful for understanding the character variety of $J(2,2n)$, see \cite{NT}.
\end{remark}



\section{Non-abelian $SL_2(\BC)$-representations}

Recall that $K=J(2m,2n)$. A representation $\rho:\pi_1(K)\to SL_2(\BC)$ is called non-abelian if 
$\rho(\pi_1(K))$ is a non-abelian subgroup of $SL_2(\BC)$. 
Taking conjugation if necessary, we can assume that $\rho$ 
has the form
\begin{equation}
\rho(a)=A=\left[ \begin{array}{cc}
M & 0\\
2-y & M^{-1} \end{array} \right] \quad \text{and} \quad \rho(b)=B=\left[ \begin{array}{cc}
M & 1\\
0 & M^{-1} \end{array} \right]
\label{nonabelian}
\end{equation}
where $(M,y) \in \BC^* \times \BC$ satisfies the matrix equation $AW^n-W^nB=O$. 
Here $W=\rho(w)$. It can be easily checked that $y=\tr AB^{-1}$ holds.
Let $x=\tr A=\tr B=M+M^{-1}$.  

Let $\{S_j(t)\}_j$ be the sequence of Chebyshev polynomials defined by $S_0(t)=1,\,S_1(t)=t$, and $S_{j+1}(t)=tS_j(t)-S_{j-1}(t)$ for all integers $j$. Note that $S_{-j}(t)=-S_{j-2}(t)$. Moreover if $t=s+s^{-1}$, where $s \not= \pm 1$, then $S_j(t)=\frac{s^{j+1}-s^{-j-1}}{s-s^{-1}}$.

By \cite[Section 2]{MT}, the assignment \eqref{nonabelian} gives a non-abelian representation 
$\rho: \pi_1(K) \to SL_2(\BC)$ 
if and only if $(M,y) \in \BC^* \times \BC$ satisfies the equation 
$$\phi_{K}(x,y):= \alpha_mS_{n-1}(\beta_m)-S_{n-2}(\beta_m)=0,$$
where
\begin{eqnarray*}
\beta_m &=& 2+(y-2)(y+2-x^2)S^2_{m-1}(y),\\
\alpha_m &=& 1-(y+2-x^2)S_{m-1}(y) \left( S_{m-1}(y) - S_{m-2}(y) \right).
\end{eqnarray*}
The polynomial $\phi_{K}(x,y)$ is also known as the Riley polynomial \cite{Ri, Le93} of $K$. Certain roots of $\phi_{K}(x,y)$ can be described as follows.

\begin{lemma} Suppose $|n| \ge 2$. There are $0<\delta_1<\delta_2<4$ (depending on $n$) such that for every real $y>2$, there exists $$x \in \left( \sqrt{y+2+\frac{\delta_1}{(y-2)S^2_{m-1}(y)}}, \sqrt{y+2+\frac{\delta_2}{(y-2)S^2_{m-1}(y)}}~\right)$$ such that $\phi_K(x,y)=0.$
\label{soln}
\end{lemma}

\begin{proof}
Fix $y>2$. We consider the following 3 cases.

\textit{\underline{Case 1}: $n=2$.} We have $\phi_K(x,y)=\alpha_m\beta_m-1$. If $x=\sqrt{y+2+\frac{2}{(y-2)S^2_{m-1}(y)}}$ then $\beta_m=0$, and $\phi_K(x,y)=-1<0$. If $x=\sqrt{y+2+\frac{1}{(y-2)S^2_{m-1}(y)}}$ then $\beta_m=1$ and $\alpha_m>1$, which implies that $\phi_K(x,y)=\alpha_m-1>0$. Hence there exists $$x \in \left( \sqrt{y+2+\frac{1}{(y-2)S^2_{m-1}(y)}}, \sqrt{y+2+\frac{2}{(y-2)S^2_{m-1}(y)}} ~\right)$$ such that $\phi_K(x,y)=0$.

\textit{\underline{Case 2}: $n>2$.} It is known that the polynomial $S_{n-1}(t)-S_{n-2}(t)$ has exactly $n-1$ roots given by $t=2\cos \frac{(2j-1)\pi}{(2n-1)}$, where $1 \le j \le n-1$. 

Let $x_j=\sqrt{y+2+\frac{2-2\cos \frac{(2j-1)\pi}{2n-1}}{(y-2)S^2_{m-1}(y)}}$. Note that if $x=x_j$ then $\beta_m=2\cos \frac{(2j-1)\pi}{(2n-1)}$, which implies that $S_{n-1}(\beta_m)=S_{n-1}(\beta_m)$ and $\phi_K(x_j,y)=(\alpha_m-1)S_{n-1}(2\cos \frac{(2j-1)\pi}{(2n-1)})$. In particular, we have $\phi_K(x_1,y)>0>\phi_K(x_2,y)$, since $S_{n-1}(2\cos \frac{\pi}{2n-1})>0>S_{n-1}(2\cos \frac{3\pi}{2n-1})$ (see e.g. \cite[Lemma 3.1]{HT2}). Hence there exists $x \in (x_1,x_2)$ such that $\phi_K(x,y)=0$. 

\textit{\underline{Case 3}: $n \le -2$.} Let $l=-n  \ge 2$. We have $$\phi_{K}(x,y):=\alpha_m S_{n-1}(\beta_m)-S_{n-2}(\beta_m)=S_l(\beta_m)-\alpha_mS_{l-1}(\beta_m).$$ 
Let $x'_j=\sqrt{y+2+\frac{2-2\cos \frac{(2j-1)\pi}{2l+1}}{(y-2)S^2_{m-1}(y)}}$, where $1 \le j \le l$. By a similar argument as in the previous case, we can show that $\phi_K(x'_1,y)<0<\phi_K(x'_2,y)$. Hence there exists $x \in (x'_1,x'_2)$ such that $\phi_K(x,y)=0$.
\end{proof}

In the case of $m=1$ (twist knots), by using the presentation in Lemma \ref{knot-group} we can also describe non-abelian $SL_2(\BC)$-representations of $K=J(2,2n)$ as follows. Suppose $\rho:\pi_1(K)\to SL_2(\BC)$ is a non-abelian representation. 
Taking conjugation if necessary, we can assume that $\rho$ 
has the form
\begin{equation}
\rho(b)=B=\left[ \begin{array}{cc}
M & 1\\
0 & M^{-1} \end{array} \right]  \quad \text{and} \quad \rho(c)=C= \left[ \begin{array}{cc}
M & 0\\
2-z & M^{-1} \end{array} \right] 
\label{nonabelian1}
\end{equation}
where $(M,z) \in \BC^* \times \BC$ satisfies the matrix equation $BU-UC=O$. Here $U=\rho(u)$.

It can be easily checked that $z=\tr BC^{-1}$. The following lemma is standard.

\begin{lemma}
Suppose the sequence $\{D_j\}_j$ of $2 \times 2$ matrices satisfies 
the recurrence relation $D_{j+1}=tD_j-D_{j-1}$ for all integers $j$. Then 
\begin{eqnarray}
D_j &=& S_{j-1}(t)D_1-S_{j-2}(t)D_0 \label{0}.
\end{eqnarray}
\end{lemma}

\begin{proposition}
\label{BUUC}
One has $$BU-UC=\left[ \begin{array}{cc}
(2-z)\gamma_n(x,z) & M^{-1}\gamma_n(x,z)\\  
(z-2)M\gamma_n(x,z)
 & 0 \end{array} \right]$$
where
\begin{eqnarray*}
\gamma_n(x,z) &=& -(z+1)S^2_{n-1}(z)+S^2_{n-2}(z)+2S_{n-1}(z)S_{n-2}(z)\\
&& + \, x^2S_{n-1}(z) \left(S_{n-1}(z)-S_{n-2}(z)\right).
\end{eqnarray*}
\end{proposition}

\begin{proof}
We first note that, by the Cayley-Hamilton theorem, $D^{j+1}=(\tr D) D^j - D^{j-1}$ for all matrices $D \in SL_2(\BC)$ and all integers $j$. By applying \eqref{0} twice, we have
\begin{eqnarray*}
BU &=& B(B^{-1}C)^nC(C^{-1}B)^n\\
&=& S^2_{n-1}(z)B(B^{-1}C)C(C^{-1}B)+S^2_{n-2}(z)BC\\
&& - \, S_{n-1}(z)S_{n-2}(z)(B(B^{-1}C)C+BC(C^{-1}B))\\
&=& S^2_{n-1}(z)CB+S^2_{n-2}(z)BC-S_{n-1}(z)S_{n-2}(z)(C^2+B^2).
\end{eqnarray*}

Similarly,
\begin{eqnarray*}
UC &=& (B^{-1}C)^nC(C^{-1}B)^nC\\
&=& S^2_{n-1}(z)(B^{-1}C)C(C^{-1}B)C+ S^2_{n-2}(z)CC\\
&& - \, S_{n-1}(z)S_{n-2}(z)((B^{-1}C)CC+C(C^{-1}B)C)\\
&=& S^2_{n-1}(z)B^{-1}CBC+S^2_{n-2}(z)C^2-S_{n-1}(z)S_{n-2}(z)(B^{-1}C^3+BC).
\end{eqnarray*}
Hence, by direct calculations using \eqref{nonabelian1}, we obtain
\begin{eqnarray*}
BU-UC &=& S^2_{n-1}(z)(CB-B^{-1}CBC)+S^2_{n-2}(z)(BC-C^2)\\
&& - \, S_{n-1}(z)S_{n-2}(z)(C^2-B^{-1}C^3+B^2-BC)\\
&=& \left[ \begin{array}{cc}
(2-z)\gamma_n(x,z) & M^{-1}\gamma_n(x,z)\\  
(z-2)M\gamma_n(x,z)
 & 0 \end{array} \right]
\end{eqnarray*}
where
\begin{eqnarray*}
\gamma_n(x,z) &=& (M^2+M^{-2}+1-z)S^2_{n-1}(z)-(M^2+M^{-2})S_{n-1}(z)S_{n-2}(z)+S^2_{n-2}(z).
\end{eqnarray*}
The proposition follows since $M^2+M^{-2}=x^2-2$.
\end{proof}





Proposition \ref{BUUC} implies that the assignment \eqref{nonabelian1} gives a non-abelian representation $\rho: \pi_1(J(2,2n)) \to SL_2(\BC)$ if and only if $\gamma_n(x,z)=0$. 

\section{Canonical longitudes} 

\label{ll}

Recall that $X$ is the closure of $S^3$
minus a tubular neighborhood of a knot $K$. The boundary of
$X$ is a torus $\BT^2$. There is a standard choice of a meridian $\mu$ and a longitude $\lambda$ on $\BT^2$ such that the linking number between the longitude and the knot is 0. We call $\lambda$ the canonical longitude of $K$ corresponding to the meridian $\mu$.

Let $\mu=b$ be the meridian of $K=J(2m,2n)$ and $\lambda$ the canonical longitude corresponding to $\mu$. Suppose $\rho: \pi_1(K) \to SL_2(\BC)$ is a non-abelian representation. By taking conjugation if necessary, we can assume that $\rho$ 
has the form
\begin{equation*}
\rho(a)=A= \left[ \begin{array}{cc}
M & 0\\
2-y & M^{-1} \end{array} \right] \quad \text{and} \quad \rho(b)=B=\left[ \begin{array}{cc}
M & 1\\
0 & M^{-1} \end{array} \right]
\end{equation*}
where $y=\tr AB^{-1}$. Recall that $x=\tr A=\tr B=M+M^{-1}$.


By \cite[Section 4]{HSn}, we have $\rho(\lambda)=\left[ \begin{array}{cc}
L & *\\
0 & L^{-1} \end{array} \right]$ where $L=-\widetilde{W}_{12}/W_{12}$. Here $W_{ij}$ is the $ij$-entry of $W=\rho(w)$ and $\widetilde{W}_{ij}$ is obtained from $W_{ij}$ by replacing $M$ by $M^{-1}$.

\begin{lemma}
\label{w12}
One has $$W_{12}=S_{m-1}(y) \left[ xS_{m-1}(y) - (M-M^{-1}) S_{m-2}(y)-yM^{-1}S_{m-1}(y) \right].$$
\end{lemma}

\begin{proof}
The proof is similar to that of \cite[Lemma 2.3]{MT}, so we omit the details.
\end{proof}


In the case of $m=1$ (twist knots), by Lemma \ref{w12} we have $\rho(\lambda)=\left[ \begin{array}{cc}
L & *\\
0 & L^{-1} \end{array} \right]$ where \begin{equation}
L=\frac{1-(y-1)M^2}{y-1-M^2}.
\label{Lz'}
\end{equation} 
By Lemma \ref{knot-group}, the knot group of $J(2,2n)$ also has the following presentation $$\pi_1(J(2,2n))=\la b,c \mid bu=uc \ra$$ where $u=(b^{-1}c)^{n}c(b^{-1}c)^{-n}$. Recall from the previous section that $C=\rho(c)$ and $z=\tr BC^{-1}$. We can express $y=\tr AB^{-1}$ in terms of $x$ and $z$ as follows.

\begin{lemma}
\label{trace-odd}
One has $$y=(z^2-2)S^2_{n-1}(z)+2S^2_{n-2}(z)-2zS_{n-1}(z)S_{n-2}(z)-x^2(z-2)S^2_{n-1}(z).$$
\end{lemma}

\begin{proof}
From the proof of Lemma \ref{knot-group}, we have $a=b_{|n|+1}=(b^{-1}c)^{n}b(b^{-1}c)^{-n}$, see Figures 2 and 3. By applying \eqref{0} twice, we have
\begin{eqnarray*}
AB^{-1} &=& (B^{-1}C)^{n}B(C^{-1}B)^{n}B^{-1}\\
&=& S^2_{n-1}(z)(B^{-1}C)B(C^{-1}B)B^{-1}+S^2_{n-2}(z) BB^{-1}\\
&& - \, S_{n-1}(z)S_{n-2}(z) \left( (B^{-1}C)BB^{-1}+B(C^{-1}B)B^{-1} \right)\\
&=& S^2_n(z)B^{-1}CBC^{-1}+S^2_{n-1}(z) I-S_{n-1}(z)S_{n-2}(z) \left( B^{-1}C+BC^{-1} \right),
\end{eqnarray*}
where $I$ is the $2 \times 2$ identity matrix. Taking traces, we obtain 
\begin{eqnarray*}
\tr AB^{-1} &=& S^2_{n-1}(z) \tr(B^{-1}CBC^{-1})+2S^2_{n-2}(z)-2zS_{n-1}(z)S_{n-2}(z)\\
&=& (z^2-zx^2+2x^2-2)S^2_{n-1}(z)+2S^2_{n-2}(z)-2zS_{n-1}(z)S_{n-2}(z),
\end{eqnarray*}
since $\tr(B^{-1}CBC^{-1})=z^2-zx^2+2x^2-2$. The lemma follows.
\end{proof}

In Sections \ref{lll} and \ref{m=1} below we will perform crucial calculations involving the meridian and the canonical longitude of the knot $J(2m,2n)$ which will be needed in the proof of Theorem \ref{main} in the last section. 

\section{Calculations: The case of $|n| \ge 2$}
\label{lll}

Recall that $K=J(2m,2n)$. Let $s>1$ and $y=s+s^{-1}$. By Lemma \ref{soln}, there exists $$x \in \left( \sqrt{y+2+\frac{\delta_1}{(y-2)S^2_{m-1}(y)}}, \sqrt{y+2+\frac{\delta_2}{(y-2)S^2_{m-1}(y)}}~\right)$$ such that $\phi_K(x,y)=0$, where $0<\delta_1<\delta_2<4$ depending on $n$ only. Since $x>\sqrt{y+2}>2$, there exists $M_s>1$ such that $x=M_s+M_s^{-1}$. Because $\phi_K(x,y)=0$, there exists a non-abelian representation $\rho_s: \pi_1(K) \to SL_2(\BR)$ of the form
\begin{equation*}
\rho_s(a)=A=\left[ \begin{array}{cc}
M_s & 0\\
2-y & M_s^{-1} \end{array} \right] \quad \text{and} \quad \rho_s(b)=B=\left[ \begin{array}{cc}
M_s & 1\\
0 & M_s^{-1} \end{array} \right].
\end{equation*} 

Recall from the previous section that $\mu=b$ is the meridian of $K$ and $\lambda$ is the canonical longitude corresponding to $\mu$. We have $\rho_s(\lambda)=\left[ \begin{array}{cc}
L_s & *\\
0 & L_s^{-1} \end{array} \right]$ where 
\begin{eqnarray*}
L_s &=& -\frac{\widetilde{W}_{12}}{W_{12}} = -\frac{xS_{m-1}(y) + (M-M^{-1}) S_{m-2}(y)-yMS_{m-1}(y)}{xS_{m-1}(y) - (M-M^{-1}) S_{m-2}(y)-yM^{-1}S_{m-1}(y)}\\
&=&\frac{M^2-s-s^{2m}+M^2s^{1+2m}}{-1+M^2s+M^2s^{2m}-s^{1+2m}}
\end{eqnarray*}
by Lemma \ref{w12}.

\begin{lemma}
One has $M_s^2>s>1$. Hence $L_s>1$.
\label{lem1a}
\end{lemma}

\begin{proof}
We have $x^2>y+2$, or equivalently $M_s^2+M_s^{-2}+2>s+s^{-1}+2$. It follows that $M_s^2>s>1$, and hence $L_s>1$.
\end{proof}

\begin{lemma}
One has $\lim_{s \to 1^+} \left(\frac{\log L_s}{\log M_s}\right)=0$ and $\lim_{s \to \infty} \left(\frac{\log L_s}{\log M_s}\right)=4m$.
\label{lem2a}
\end{lemma}

\begin{proof}
Let $s \to \infty$. Since $x^2 \in \left( y+2+\frac{\delta_1}{(y-2)S^2_{m-1}(y)}, y+2+\frac{\delta_2}{(y-2)S^2_{m-1}(y)}\right)$, we have $x^2-(y+2) \to 0$, or equivalently $(M_s^2-s)(1-\frac{1}{sM_s^2}) \to 0$. It follows that $M^2-s \to 0$, and $$L-s^{2m}=\frac{M^2-s-s^{2m}+M^2s^{1+2m}}{-1+M^2s+M^2s^{2m}-s^{1+2m}}-s^{2m} \to 0.$$ Hence $\lim_{s \to \infty} \left(\frac{\log L_s}{\log M_s}\right)=4m$.

Let $s \to 1^+$, $y \to 2^+$. Since $x^2 \in \left( y+2+\frac{\delta_1}{(y-2)S^2_{m-1}(y)}, y+2+\frac{\delta_2}{(y-2)S^2_{m-1}(y)}\right)$, we have $x^2 \to \infty$. It follows that $M_s \to \infty$ and $$L_s=\frac{M^2-s-s^{2m}+M^2s^{1+2m}}{-1+M^2s+M^2s^{2m}-s^{1+2m}} \to 1.$$ Hence $\lim_{s \to 1^+} \left(\frac{\log L_s}{\log M_s} \right)=0$.
\end{proof}

Let $f_0: (1, \infty) \to \BR$ be the function defined by $f_0(s)=-\dfrac{\log L_s}{\log M_s}.$ 
Lemmas \ref{lem1a} and \ref{lem2a} imply the following.
\begin{proposition}
The image of $f_0$ contains the interval $(-4m,0)$.
\label{image}
\end{proposition}

\section{Calculations: The case of $m=1$} 
\label{m=1}
Let $K=J(2,2n)$. Recall from Proposition \ref{BUUC} and Lemma \ref{trace-odd} that
\begin{eqnarray*}
\gamma_n(x,z) &=& -(z+1)S^2_{n-1}(z)+S^2_{n-2}(z)+2S_{n-1}(z)S_{n-2}(z)\\
&& + \, x^2S_{n-1}(z) \left(S_{n-1}(z)-S_{n-2}(z)\right)\\
y &=&(z^2-2)S^2_{n-1}(z)+2S^2_{n-2}(z)-2zS_{n-1}(z)S_{n-2}(z)-x^2(z-2)S^2_{n-1}(z).
\end{eqnarray*}

Let $s \in \BC \setminus \{-1,0,1\}$ and $z=s+s^{-1}$. Note that $S_j(z)=\frac{s^{j+1}-s^{-j-1}}{s-s^{-1}}$ for all integers $j$.

\begin{lemma}
\label{b}
Suppose $(s^{2n}-1)(s^{2n-1}+1)s \not=0$ and $x^2=(2+s+s^{-1})\frac{s^{4n-1}-1}{(s^{2n}-1)(s^{2n-1}+1)}$. Then $\gamma_n(x,z)=0$ and $y-1=\frac{s^{2n+1}+1}{s^{2n}+s}$.
\end{lemma}

\begin{proof}
Since $z=s+s^{-1}$, by direct calculations, we have
\begin{eqnarray*}
-(z+1)S^2_{n-1}(z)+S^2_{n-2}(z)+2S_{n-1}(z)S_{n-2}(z) &=& -\frac{s^{4n-1}-1}{s^{2n-1}(s-1)},\\
S_{n-1}(z) \left(S_{n-1}(z)-S_{n-2}(z)\right)  &=& \frac{(s^{2n-1}+1)(s^{2n}-1)}{s^{2n-2}(s-1)(s+1)^2}.
\end{eqnarray*}
By assumption, $x^2=(2+s+s^{-1})\frac{s^{4n-1}-1}{(s^{2n}-1)(s^{2n-1}+1)}$. It follows that $\gamma_n(x,z)=0$.

Similarly, $y-1=\frac{s^{2n+1}+1}{s^{2n}+s}$ by direct calculations.
\end{proof}

\subsection{The case of $n>0$}

\begin{lemma}
On the real interval $(1, \infty)$, the equation $(2+s+s^{-1})\frac{s^{4n-1}-1}{(s^{2n}-1)(s^{2n-1}+1)}=4$ has a unique solution $s_0$ .
\label{unique}
\end{lemma}

\begin{proof}
Suppose $s$ is a real number $>1$. Then the equation is equivalent to $\frac{(s^{2n}-1)(s^{2n-1}+1)}{s^{4n-1}-1}=\frac{(s+1)^2}{4s}$, i.e. $\frac{s^{2n}-s^{2n-1}}{s^{4n-1}-1}=\frac{(s-1)^2}{4s}$, or equivalently $(s^{2n-1}-s^{-2n})(s-1)=4$. The $LHS=(s^{2n-1}-s^{-2n})(s-1)$ is a strictly increasing function in $s>1$. Hence the lemma follows since $\lim_{s \to 1^+} LHS=0<4<\infty=\lim_{s \to \infty} LHS$.
\end{proof}

\subsubsection{The case of $s>s_0$} Suppose $s>s_0$. Since 
$$(2+s+s^{-1})\frac{s^{4n-1}-1}{(s^{2n}-1)(s^{2n-1}+1)}>4$$
by Lemma \ref{unique}, there exists $x>2$ such that $x^2=(2+s+s^{-1})\frac{s^{4n-1}-1}{(s^{2n}-1)(s^{2n-1}+1)}$. By Lemma \ref{b}, $\gamma_n(x,z)=0$. 

Choose $M_s>1$ such that $x=M_s+M_s^{-1}$. Since $\gamma_n(x,z)=0$, Proposition \ref{BUUC} implies that there exists a non-abelian representation $\rho_s: \pi_1(K) \to SL_2(\BR)$ satisfying
\begin{equation*}
\rho_s(a)=A= \left[ \begin{array}{cc}
M_s & 0\\
2-y & M_s^{-1} \end{array} \right] \quad \text{and} \quad \rho_s(b)=B=\left[ \begin{array}{cc}
M_s & 1\\
0 & M_s^{-1} \end{array} \right]
\end{equation*}
where $y=\tr AB^{-1}=1+\frac{s^{2n+1}+1}{s^{2n}+s}$ by Lemmas \ref{trace-odd} and \ref{b}.

By \eqref{Lz'}, we have $\lambda=\left[ \begin{array}{cc}
L_s & *\\
0 & L_s^{-1} \end{array} \right]$ where $L_s=\frac{1-(y-1)M_s^2}{y-1-M_s^2}$.

\begin{lemma}
One has
\begin{equation}
(2+s+s^{-1}) \frac{s^{4n-1}-1}{(s^{2n}-1)(s^{2n-1}+1)}<\frac{s^{2n+1}+1}{s^{2n}+s}+\frac{s^{2n}+s}{s^{2n+1}+1}+2.
\label{ineq5}
\end{equation}
\end{lemma}

\begin{proof}
Since 
$$LHS-RHS=\frac{-(s+1)^2(s^{2n}-s)}{(s^{2n+1}+1)(s^{2n}-1)}<0,$$
the lemma follows.
\end{proof}

\begin{lemma}
One has $y-1>M_s^2>1$. Hence $L_s<-1$.
\label{-}
\end{lemma}

\begin{proof}
We have $y-1=\frac{s^{2n+1}+1}{s^{2n}+s}>1$. The inequality \eqref{ineq5} is equivalent to $M_s^2+M_s^{-2}<y-1+\frac{1}{y-1}$. It follows that $y-1>M_s^2>1$ and
$L_s =\frac{1-(y-1)M_s^2}{y-1-M_s^2}<-1.$
\end{proof}

\begin{lemma}
One has $\lim_{s \to \infty} \left(\frac{\log |L_s|}{\log M_s^2}\right)=2n+1$.
\label{oo}
\end{lemma}

\begin{proof}
We have $$M_s^2+M_s^{-2}=x^2-2=s+s^{-1}-(2+s+s^{-1}) \frac{s^{2n-1}(s-1)}{(s^{2n}-1)(s^{2n-1}+1)}.$$
It follows that
\begin{eqnarray*}
M_s^{2} &=& \frac{1}{2} \left( s+s^{-1}-(2+s+s^{-1}) \frac{s^{2n-1}(s-1)}{(s^{2n}-1)(s^{2n-1}+1)} \right)\\
&& + \, \frac{1}{2} \sqrt{\left( s+s^{-1}-(2+s+s^{-1}) \frac{s^{2n-1}(s-1)}{(s^{2n}-1)(s^{2n-1}+1)} \right)^2-4}.
\end{eqnarray*}

It is easy to show that 
\begin{eqnarray*}
\lim_{s \to \infty} (s+s^{-1}-s^{2-2n}-s^{1-2n})^{-1} \left( s+s^{-1}-(2+s+s^{-1}) \frac{s^{2n-1}(s-1)}{(s^{2n}-1)(s^{2n-1}+1)} \right)  &=&1,\\
\lim_{s \to \infty} \left( s-s^{-1}-s^{2-2n}-s^{1-2n} \right)^{-1}\sqrt{\left( s+s^{-1}-(2+s+s^{-1}) \frac{s^{2n-1}(s-1)}{(s^{2n}-1)(s^{2n-1}+1)} \right)^2-4}&=&1.
\end{eqnarray*}
Hence $\lim_{s \to \infty} \left( s-s^{2-2n}-s^{1-2n} \right)^{-1}M_s^{2}=1$ and $\lim_{s \to \infty} \left( M_s^{2} - \frac{s^{2n+1}+1 }{s^{2n}+s} \right) \big/ s^{1-2n}=-1$. Since 
$L_s =  \left( \frac{s^{2n+1}+1}{s^{2n}+s} M_s^2-1 \right) \big/ \left( M_s^2 - \frac{s^{2n+1}+1}{s^{2n}+s} \right),$
we have $\lim_{s \to \infty} s^{-2n-1}L_s=-1$. The lemma follows.
\end{proof}

Let $\omega>1$ be the unique real solution of the equation $se^{s}=4(2n-1)$ satisfying $s>1$.

\begin{lemma}
One has $\lim_{s \to s_0^+} \left(\frac{\log |L_s|}{\log M_s^2}\right)<\frac{2(2n-1)}{\omega}+2$.
\label{s0}
\end{lemma}

\begin{proof}
From the proof of Lemma \ref{unique}, it follows that $s_0>1$ is the solution of $(s^{4n-1}-1)(s-1)=4s^{2n}$, or equivalently $(s^{2n}-1)^2=s(s^{2n-1}+1)^2$. Hence $\frac{s_0^{2n}-1}{s_0^{2n-1}+1}=\sqrt{s_0}$ and
$$\lim_{s \to s_0^+} y-1=\lim_{s \to s_0^+} \frac{s^{2n+1}+1}{s^{2n}+s}=\lim_{s \to s_0^+} 1+\frac{(s-1)(s^{2n}-1)}{s(s^{2n-1}+1)}=1+\frac{s_0-1}{\sqrt{s_0}}.$$
Let $\gamma=1+\frac{s_0-1}{\sqrt{s_0}}$. By L'Hospital's rule, we have
\begin{eqnarray*}
\lim_{s \to s_0^+} \left(\frac{\log |L_s|}{\log M_s^2}\right) &=& \lim_{t=M_s^2 \to 1^+} \frac{\log (\gamma t-1)-\log (\gamma-t)}{\log t} = \frac{\gamma +1}{\gamma -1}=1+\frac{2}{\gamma-1}.
\end{eqnarray*}

We claim that $s_0 > 1+\frac{\omega}{2n-1}$. Indeed, assume that $s_0 \le 1+\frac{\omega}{2n-1}$.  Then
$$4=(s_0^{2n-1}-s_0^{-2n})(s_0-1)<s_0^{2n-1}(s_0-1) \le \left( 1+\frac{\omega}{2n-1} \right)^{2n-1}\frac{\omega}{2n-1}< e^\omega \frac{\omega}{2n-1}=4,$$
a contradiction. Hence $s_0 > 1+\frac{\omega}{2n-1}$ and $$\gamma-1=\frac{s_0-1}{\sqrt{s_0}} > \frac{\frac{\omega}{2n-1}}{\sqrt{1+\frac{\omega}{2n-1}}}=\frac{\omega}{\sqrt{(2n-1)(2n-1+\omega)}}>\frac{2\omega}{4n-2+\omega}.$$
Therefore $\lim_{s \to s_0^+} \left(\frac{\log |L_s|}{\log M_s^2}\right)=1+\frac{2}{\gamma-1}<1+\frac{4n-2+\omega}{\omega}=\frac{2(2n-1)}{\omega}+2$.
\end{proof}


Let $f_1: (s_0, \infty) \to \BR$ be the function defined by $f_1(s)=-\dfrac{\log |L_s|}{\log M_s}.$ Lemmas \ref{-}, \ref{oo} and \ref{s0} imply the following.

\begin{proposition}
The image of $f_1$ contains the interval $(-(4n+2),-(\frac{4(2n-1)}{\omega}+4))$.
\label{kq3}
\end{proposition}



\subsubsection{The case of $s=e^{2\theta i}$} Then $z=2\cos 2\theta$ and
\begin{eqnarray*}
(2+s+s^{-1}) \frac{s^{4n-1}-1}{(s^{2n}-1)(s^{2n-1}+1)} &=& 
\frac{4\cos^2\theta \sin (4n-1)\theta}{2  \sin (2n) \theta \cos (2n-1)\theta},\\
\frac{s^{2n+1}+1}{s^{2n}+s}&=&\frac{\cos (2n+1)\theta}{\cos (2n-1)\theta}.
\end{eqnarray*}

Suppose $n>1$. Consider $\frac{\pi}{2(2n-1)}<\theta<\frac{\pi}{2n}$. 

\begin{lemma}
One has
\begin{eqnarray}
\frac{4\cos^2\theta \sin (4n-1)\theta}{2\cos (2n-1)\theta  \sin (2n) \theta}>\frac{\cos (2n-1)\theta}{\cos (2n+1)\theta}+\frac{\cos (2n+1)\theta}{\cos (2n-1)\theta}+2.
\label{ineq1}
\end{eqnarray}
\label{theta_1}
\end{lemma} 

\begin{proof}
We have 
\begin{eqnarray*}
LHS-RHS &=& \frac{2\cos^2\theta}{\cos (2n-1)\theta} \left( \frac{\sin (4n-1)\theta}{\sin (2n \theta)} - \frac{2\cos^2 (2n\theta)}{\cos (2n+1)\theta}\right)\\
&=& \frac{-2\cos^2\theta\sin\theta}{\sin (2n \theta)\cos (2n+1)\theta}>0.
\end{eqnarray*}
The lemma follows.
\end{proof}

We have $\cos (2n-1)\theta-\cos (2n+1)\theta=2\sin\theta \sin (2n\theta)>0$. It follows that $\cos (2n+1)\theta<\cos (2n-1)\theta<0$ and $\frac{\cos (2n+1)\theta}{\cos (2n-1)\theta}>1$. Lemma \ref{theta_1} implies that $$\frac{4\cos^2\theta \sin (4n-1)\theta}{2\cos (2n-1)\theta  \sin (2n) \theta}>\frac{\cos (2n-1)\theta}{\cos (2n+1)\theta}+\frac{\cos (2n+1)\theta}{\cos (2n-1)\theta}+2>4.$$
Hence there exists $x>2$ such that $$x^2=\frac{4\cos^2\theta \sin (4n-1)\theta}{2  \sin (2n) \theta \cos (2n-1)\theta}=(2+s+s^{-1}) \frac{s^{4n-1}-1}{(s^{2n}-1)(s^{2n-1}+1)}.$$ By Lemma \ref{b}, $\gamma_n(x,z)=0$. 

Choose $M_\theta>1$ such that $x=M_\theta+M_\theta^{-1}$. Since $\gamma_n(x,z)=0$, Proposition \ref{BUUC} implies that there exists a non-abelian representation $\rho_\theta: \pi_1(K) \to SL_2(\BR)$ satisfying
\begin{equation*}
\rho_\theta(a)=A= \left[ \begin{array}{cc}
M_\theta & 0\\
2-y & M_\theta^{-1} \end{array} \right] \quad \text{and} \quad \rho_\theta(b)=B=\left[ \begin{array}{cc}
M_\theta & 1\\
0 & M_\theta^{-1} \end{array} \right].
\end{equation*}
where $y=\tr AB^{-1}=1+\frac{s^{2n+1}+1}{s^{2n}+s}=1+\frac{\cos (2n+1)\theta}{\cos (2n-1)\theta}$ by Lemmas \ref{trace-odd} and \ref{b}.

By \eqref{Lz'}, we have $\lambda=\left[ \begin{array}{cc}
L_\theta & *\\
0 & L_\theta^{-1} \end{array} \right]$ where $L_\theta=\frac{1-(y-1)M_\theta^2}{y-1-M_\theta^2}$.

\begin{lemma}
One has $M_\theta^2>y-1>1$. Hence $L_\theta>1$.
\label{lem10}
\end{lemma}

\begin{proof}
We have $y-1=\frac{\cos (2n+1)\theta}{\cos (2n-1)\theta}>1$. The inequality \eqref{ineq1} is equivalent to $M_\theta^2+M_\theta^{-2}+2>y-1+\frac{1}{y-1}+2$. It follows that $M_\theta^2>y-1>1$ and $L_\theta=\frac{1-(y-1)M_\theta^2}{y-1-M_\theta^2}>1$.
\end{proof}

\begin{lemma}
One has $\displaystyle\lim_{\theta \to \left(\frac{\pi}{2(2n-1)}\right)^+} \left(\frac{\log L_\theta}{\log M_\theta^2}\right)=2$ and $\displaystyle\lim_{\theta \to \left(\frac{\pi}{2n}\right)^-} \left(\frac{\log L_\theta}{\log M_\theta^2}\right)=0$.
\label{lem20}
\end{lemma}

\begin{proof}
For the first limit, let $\theta_1=\frac{\pi}{2(2n-1)}$. Since 
$$\lim_{\theta \to \theta_1^+} \left( \frac{-2\cos^2\theta\sin\theta}{\sin (2n \theta)\cos (2n+1)\theta} \right)=\frac{-2\cos^2\theta_1\sin\theta_1}{\cos \theta_1 (-\sin 2\theta_1)}=1,$$
the proof of Lemma \ref{lem10} implies that $\lim_{\theta \to \theta_1^+} (M_\theta^2+M_\theta^{-2})-\left(y-1+\frac{1}{y-1} \right)=1$. Hence $\lim_{\theta \to \theta_1^+}M_\theta^2-(y-1)=1$ and $$\lim_{\theta \to \theta_1^+} \left(\frac{\log L_\theta}{\log M_\theta^2}\right)=\lim_{\theta \to \theta_1^+} \frac{\log ((y-1)M_\theta^2-1) - \log (M_\theta^2-(y-1))}{\log M_\theta^2}=2.$$
The second limit is clear, since $M^2_\theta \to \infty$ and $L_\theta \to 1$ as $\theta \to \left(\frac{\pi}{2n}\right)^-$. 
\end{proof}

Let $f_2: (\frac{\pi}{2(2n-1)},\frac{\pi}{2n}) \to \BR$ be the function defined by $f_2(\theta)=-\dfrac{\log L_\theta}{\log M_\theta}.$ Lemmas \ref{lem10} and \ref{lem20} imply the following.

\begin{proposition}
The image of $f_2$ contains the interval $(-4,0)$.
\label{kq4}
\end{proposition}

\subsection{The case of $n<0$} Let $l=-n>0$. From Lemma \ref{b}, we have

\begin{lemma}
Suppose $(s^{2l+1}+1)(s^{2l}-1)s \not=0$ and $x^2=(2+s+s^{-1}) \frac{s^{4l+1}-1}{(s^{2l+1}+1)(s^{2l}-1)}$. Then $\gamma_{n}(x,z)=0$ and $y-1=\frac{s^{2l}+s}{s^{2l+1}+1}$.
\label{s>1}
\end{lemma}

\subsubsection{The case of $s>1$} Suppose $s>1$. Since 
$$(2+s+s^{-1}) \frac{s^{4l+1}-1}{(s^{2l+1}+1)(s^{2l}-1)}=(2+s+s^{-1}) \left( 1+ \frac{s^{2l}(s-1)}{(s^{2l+1}+1)(s^{2l}-1)} \right)>4,$$
there exists $x>2$ such that $x^2=(2+s+s^{-1}) \frac{s^{4l+1}-1}{(s^{2l+1}+1)(s^{2l}-1)}$. By Lemma \ref{s>1}, $\gamma_{n}(x,z)=0$. 

Choose $M_s>1$ such that $x=M_s+M_s^{-1}$. Since $\gamma_{n}(x,z)=0$, Proposition \ref{BUUC} implies that there exists a non-abelian representation $\rho_s: \pi_1(K) \to SL_2(\BR)$ satisfying
\begin{equation*}
\rho_s(a)=A= \left[ \begin{array}{cc}
M_s & 0\\
2-y & M_s^{-1} \end{array} \right] \quad \text{and} \quad \rho_s(b)=B=\left[ \begin{array}{cc}
M_s & 1\\
0 & M_s^{-1} \end{array} \right]
\end{equation*}
where $y=\tr AB^{-1}=1+\frac{s^{2l}+s}{s^{2l+1}+1}$ by Lemmas \ref{trace-odd} and \ref{s>1}.

By \eqref{Lz'}, we have $\lambda=\left[ \begin{array}{cc}
L_s & *\\
0 & L_s^{-1} \end{array} \right]$ where 


\begin{eqnarray*}
L_s &=& \frac{1-(y-1)M_s^2}{y-1-M_s^2}= \left(  \frac{s^{2l}+s}{s^{2l+1}+1} M_s^2 - 1\right) \big/ \left( M_s^2 - \frac{s^{2l}+s}{s^{2l+1}+1}\right).
\end{eqnarray*}

\begin{lemma}
One has $M^2_s>s$. Hence $0<L_s<1$.
\label{01}
\end{lemma}

\begin{proof}
We have $$M_s^2+M_s^{-2}=x^2-2=s+s^{-1}+(2+s+s^{-1})\frac{s^{2l}(s-1)}{(s^{2l+1}+1)(s^{2l}-1)}.$$
It follows that 
\begin{eqnarray*}
M_s^2 &=& \frac{1}{2} \left( s+s^{-1}+(2+s+s^{-1})\frac{s^{2l}(s-1)}{(s^{2l+1}+1)(s^{2l}-1)} \right)\\
    && + \, \frac{1}{2} \sqrt{\left( s+s^{-1}+(2+s+s^{-1})\frac{s^{2l}(s-1)}{(s^{2l+1}+1)(s^{2l}-1)} \right)^2-4}\\
    &>& \frac{1}{2}(s+s^{-1})+\frac{1}{2} \sqrt{(s+s^{-1})^2-4}=s>1.
\end{eqnarray*}  
Since $M_s^2>s>\frac{s^{2l+1}+1 }{s^{2l}+s}>1>\frac{s^{2l}+s}{s^{2l+1}+1}$, we obtain $0<L_s<1$.
\end{proof}

The following lemma is easy to check.

\begin{lemma}
One has $\lim_{s \to 1^+} M_s^2=1+\frac{1+\sqrt{4l+1}}{2l}$ and $\lim_{s \to 1^+} L_s=1$. 
\label{1}
\end{lemma}


\begin{lemma}
One has $\lim_{s \to \infty} \frac{M_s^2}{s+s^{1-2l}}=1$ and 
$\lim_{s \to \infty} s^{2l}L_s=1$.
\label{infty}
\end{lemma}

\begin{proof}
It is easy to show that 
\begin{eqnarray*}
\lim_{s \to \infty} (s+s^{-1}+s^{1-2l})^{-1} \left( s+s^{-1}+(2+s+s^{-1})\frac{s^{2l}(s-1)}{(s^{2l+1}+1)(s^{2l}-1)} \right)  &=&1,\\
\lim_{s \to \infty} \left( s-s^{-1}+s^{1-2l} \right)^{-1}\sqrt{\left( s+s^{-1}+(2+s+s^{-1})\frac{s^{2l}(s-1)}{(s^{2l+1}+1)(s^{2l}-1)} \right)^2-4}&=&1.
\end{eqnarray*}
Hence $\lim_{s \to \infty} \left( s+s^{1-2l} \right)^{-1}M_s^2=1$ and $\lim_{s \to \infty} \left( M_s^2 - \frac{s^{2l+1}+1 }{s^{2l}+s} \right) \big/ s^{2-2l}=1$. Then, from 
$$L_s=\left(  \frac{s^{2l}+s}{s^{2l+1}+1} M_s^2 - 1\right) \big/ \left( M_s^2 - \frac{s^{2l}+s}{s^{2l+1}+1}\right)$$
we obtain $\lim_{s \to \infty} s^{2l}L_s=1$.
\end{proof}


Let $f_3: (1, \infty) \to \BR$ be the function defined by $f_3(s)=-\dfrac{\log L_s}{\log M_s}.$ Lemmas \ref{01}, \ref{1} and \ref{infty} imply the following.

\begin{proposition}
The image of $f_3$ contains the interval $(0,-4n)$.
\label{kq1}
\end{proposition}

\subsubsection{The case of $s=e^{2\theta i}$} 

Suppose $s=e^{2\theta i}$. Then $z=s+s^{-1}=2\cos 2\theta$. By direct calculations, we have
\begin{eqnarray*}
(2+s+s^{-1}) \frac{s^{4l+1}-1}{(s^{2l+1}+1)(s^{2l}-1)} 
&=& \frac{4\cos^2\theta \sin (4l+1)\theta}{2\cos (2l+1)\theta  \sin (2l) \theta},\\
\frac{s^{2l}+s}{s^{2l+1}+1} &=& \frac{\cos (2l-1)\theta}{\cos (2l+1)\theta}.
\end{eqnarray*}

Let $\theta_2=\frac{\pi}{2(2l+1)}$. Consider $0<\theta<\theta_2$. 

\begin{lemma}
One has
\begin{eqnarray}
\frac{4\cos^2\theta \sin (4l+1)\theta}{2\cos (2l+1)\theta  \sin (2l) \theta}>\frac{\cos (2l-1)\theta}{\cos (2l+1)\theta}+\frac{\cos (2l+1)\theta}{\cos (2l-1)\theta}+2.
\label{ineq}
\end{eqnarray}
\label{theta_0}
\end{lemma} 

\begin{proof}
We have $$RHS=\frac{(\cos (2l-1)\theta+\cos (2l+1)\theta)^2}{\cos (2l-1)\theta\cos (2l+1)\theta}=\frac{4\cos^2\theta\cos^2 (2l\theta)}{\cos (2l-1)\theta\cos (2l+1)\theta}.$$
It follows that 
\begin{eqnarray*}
LHS-RHS &=& \frac{2\cos^2\theta}{\cos (2l+1)\theta} \left( \frac{\sin (4l+1)\theta}{\sin (2l \theta)} - \frac{2\cos^2 (2l\theta)}{\cos (2l-1)\theta}\right)\\
&=& \frac{2\cos^2\theta\sin\theta}{\sin (2l \theta)\cos (2l-1)\theta}>0.
\end{eqnarray*}
The lemma follows.
\end{proof}

Since $0<(2l-1)\theta<(2l+1)\theta<\frac{\pi}{2}$, we have $\cos(2l-1)\theta>\cos(2l+1)\theta>0.$ Lemma \ref{theta_0} implies that $$\frac{4\cos^2\theta \sin (4l+1)\theta}{2\cos (2l+1)\theta  \sin (2l) \theta}>\frac{\cos (2l-1)\theta}{\cos (2l+1)\theta}+\frac{\cos (2l+1)\theta}{\cos (2l-1)\theta}+2>4.$$
Hence there exists $x>2$ such that $$x^2=\frac{4\cos^2\theta \sin (4l+1)\theta}{2\cos (2l+1)\theta  \sin (2l) \theta}=(2+s+s^{-1}) \frac{s^{4l+1}-1}{(s^{2l+1}+1)(s^{2l}-1)}.$$ By Lemma \ref{s>1}, $\gamma_n(x,z)=0$. 

Choose $M_\theta>1$ such that $x=M_\theta+M_\theta^{-1}$. Since $\gamma_n(x,z)=0$, Proposition \ref{BUUC} implies that there exists a non-abelian representation $\rho_\theta: \pi_1(K) \to SL_2(\BR)$ satisfying
\begin{equation*}
\rho_\theta(a)=A= \left[ \begin{array}{cc}
M_\theta & 0\\
2-y & M_\theta^{-1} \end{array} \right] \quad \text{and} \quad \rho_\theta(b)=B=\left[ \begin{array}{cc}
M_\theta & 1\\
0 & M_\theta^{-1} \end{array} \right].
\end{equation*}
where $y=\tr AB^{-1}=1+\frac{s^{2l}+s}{s^{2l+1}+1}=1+\frac{\cos (2l-1)\theta}{\cos (2l+1)\theta}$ by Lemmas \ref{trace-odd} and \ref{s>1}.

By \eqref{Lz'}, we have $\lambda=\left[ \begin{array}{cc}
L_\theta & *\\
0 & L_\theta^{-1} \end{array} \right]$ where $L_\theta=\frac{1-(y-1)M_\theta^2}{y-1-M_\theta^2}.$


\begin{lemma}
One has $M_\theta^2>y-1>1$. Hence $L_\theta>1$.
\label{lem1}
\end{lemma}

\begin{proof}
We have $y-1=\frac{\cos (2l-1)\theta}{\cos (2l+1)\theta}>1$. The inequality \eqref{ineq} is equivalent to $M_\theta^2+M_\theta^{-2}+2>y-1+\frac{1}{y-1}+2$. Hence $M_\theta^2>y-1>1$ and $L_\theta=\frac{1-(y-1)M_\theta^2}{y-1-M_\theta^2}>1$. 
\end{proof}

\begin{lemma}
One has $\lim_{\theta \to \theta_2^-} \left(\frac{\log L_\theta}{\log M_\theta^2}\right)=2$ and $\lim_{\theta \to 0^+} \left(\frac{\log L_\theta}{\log M_\theta^2}\right)=0$.
\label{lem2}
\end{lemma}

\begin{proof}
For the first limit, we have
$$\lim_{\theta \to \theta_2^-} \frac{2\cos^2\theta\sin\theta}{\sin (2l \theta)\cos (2l-1)\theta}=\frac{2\cos^2\theta_2\sin\theta_2}{\cos \theta_2 \sin 2\theta_2}=1.$$
The proof of Lemma \ref{theta_0} then implies that $\lim_{\theta \to \theta_2^-} (M_\theta^2+M_\theta^{-2})-\left(y-1+\frac{1}{y-1} \right)=1$. Hence $\lim_{\theta \to \theta_2^-}M_\theta^2-(y-1)=1$ and 
\begin{eqnarray*}
\lim_{\theta \to \theta_2^-} \left(\frac{\log L_\theta}{\log M_\theta^2}\right) &=& \lim_{\theta \to \theta_2^-} \frac{\log ((y-1)M_\theta^2-1)-\log (M_\theta^2-(y-1))}{\log M_\theta^2}
\\
&=& \lim_{t=M_\theta^2 \to \infty} \frac{\log ((t-1)t-1)}{\log t}=2.
\end{eqnarray*}
The second limit follows from Lemma \ref{1}.
\end{proof}



Let $f_4: (0, \frac{\pi}{2(2l+1)}) \to \BR$ be the function defined by $f_4(\theta)=-\dfrac{\log L_\theta}{\log M_\theta}.$ Lemmas \ref{lem1} and \ref{lem2} imply the following.

\begin{proposition}
The image of $f_4$ contains the interval $(-4,0)$.
\label{kq2}
\end{proposition}

\section{Proof of Theorem \ref{main}}
\label{proof-main}

Let $X_{m,n}$ be the closure of $S^3$ minus a tubular neighborhood of the knot $J(2m,2n)$. Here $m>0$ and $|n|>0$. Let $\mu$ and $\lambda$ be the pair of the meridian and the canonical longitude of $J(2m,2n)$ as defined in Section \ref{ll}.

For $r \in \BQ$, let $M_{m,n}(r)$ denote the resulting manifold by $r$-surgery on the hyperbolic knot $J(2m,2n)$. For $r=0$, $M_{m,n}(0)$ is irreducible and has positive first Betti number, so $\pi_1(M_{m,n}(0))$ is left-orderable.

\begin{lemma}
Suppose there are a continuous family of non-abelian representations $\rho_t: \pi_1(X_{m,n}) \to PSL_2(\BR), \, t \in (t_0,t_1),$ and a continuous function $g: (t_0,t_1) \to \BR$ such that the image of $g$ contains some interval $(r_0,r_1)$ and $g(t)=r \in \BQ$ if and only if $\rho_t(\mu^p\lambda^q)=\pm I$ where $r=p/q$ is a reduced fraction. Then $M_{m,n}(r)$ has left-orderable fundamental group if $r \in \BQ \cap (r_0,r_1)$.
\label{lift}
\end{lemma}

\begin{proof}
The proof is similar to that of \cite[Section 7]{BGW} and \cite[Section 7]{HT2}. The crucial point here is that the knot $J(2m,2n)$ has genus one.  

Suppose $r=p/q$ is a reduced fraction in $\BQ \cap (r_0,r_1)$. By assumption, there exists $t \in (t_0,t_1)$ such that $g(t)=r$ and $\rho_t(\mu^p\lambda^q)=\pm I$.


Let $\widetilde{SL_2}$ be the universal covering of $PSL_2(\BR)$ and $\varphi: \widetilde{SL_2} \to PSL_2(\BR)$ the covering map. It is known that there is an identification $\widetilde{SL_2} \cong \Delta \times \BR$, where $\Delta=\{ z \in \BC : |z|=1\}$, and $\ker \varphi=\{(0,j\pi) \mid j \in \BZ\}$, see e.g. \cite{Kh}.

There is a lift of $\rho_t: \pi_1(X_{m,n}) \to PSL_2(\BR)$ to a homomorphism $\widetilde{\rho_t}: \pi_1(X_{m,n}) \to \widetilde{SL_2}$ since the obstruction to its existence is the Euler class $e(\rho_t) \in H^2(X_{m,n}; \BZ) \cong 0$, see \cite{Gh}. Since the knot $J(2m,2n)$ has genus one, without loss of generality we can assume that $\widetilde{\rho_t}(\pi_1(\partial X_{m,n}))$ is contained in the subgroup $(-1,1) \times \{0\}$ of $\widetilde{SL_2}$, by \cite[Lemma 7.1]{HT2}. Because $\rho_t(\mu^p\lambda^q)=\pm I$, we have $\varphi(\widetilde{\rho_t}(\mu^p\lambda^q))=I$. This means that $\widetilde{\rho_t}(\mu^p\lambda^q)$ lies in $\ker \varphi=\{(0,j\pi) \mid j \in \BZ\}$. Hence $\widetilde{\rho_t}(\mu^p\lambda^q)=(0,0)$, the identity of $\widetilde{SL_2}$, and so $\widetilde{\rho_t}$ induces a homomorphism $\pi_1(M_{m,n}(r)) \to \widetilde{SL_2}$ with non-abelian image. Since $\widetilde{SL_2}$ is left-orderable \cite{Be}, any non-trivial subgroup of $\widetilde{SL_2}$ is left-orderable. Because $M_{m,n}(r)$ is irreducible \cite{HT}, $\pi_1(M_{m,n}(r))$ is left-orderable by \cite[Theorem 1.1]{BRW}.
\end{proof}

We are ready to prove Theorem \ref{main}. Let $r=p/q$ be a reduced fraction. Suppose $\rho: \pi_1(X_{m,n}) \to PSL_2(\BR)$ is a representation such that 
$$\rho(\mu)= \left[ \begin{array}{cc}
M & 1\\
0 & M^{-1} \end{array} \right] \quad \text{and} \quad \rho(\lambda)=\left[ \begin{array}{cc}
L & *\\
0 & L^{-1} \end{array} \right].
$$
where $M, \, L \in \BR \setminus \{0, \pm 1\}$. Since $\mu$ and $\lambda$ commute, it is easy to see that $\rho(\mu^p\lambda^q)=\pm I$ if and only if $M^{p}L^{q}=\pm I$, or equivalently $$-\frac{\log |L|}{\log |M|}=\frac{p}{q}.$$ 

We first consider $m=1$. Propositions \ref{kq3}, \ref{kq4}, \ref{kq1}, \ref{kq2} and Lemma \ref{lift} imply that $M_{m,n}(r)$ has left-orderable fundamental group if the slope $r$ satisfies the condition $$ r \in
\begin{cases}
\left(-(4n+2),-(\frac{4(2n-1)}{\omega_n}+4)\right) \cup (-4,0],& n \ge 2,\\ 
(-4,-4n),& n \le -1.\\
\end{cases}
$$ (Note that $\pi_1(M_{m,n}(0))$ is left-orderable.) Since $\pi_1(M_{1,n}(-4))$ is left-orderable by \cite{Te}, Theorem 1 follows. 

Suppose now $m \ge 2$. We consider the following cases. 

\textit{\underline{Case 1}: $n=1$.} Since $J(2m,2) \cong J(2,2m)$, $M_{m,1}(r)$ has left-orderable fundamental group if $r \in \left(-(4m+2),-(\frac{4(2m-1)}{\omega_m}+4)\right) \cup [-4,0]$.

\textit{\underline{Case 2}: $n=-1$.} Since $J(2m,-2) \cong J(-2,2m)$ is the mirror image of $J(2,-2m)$, $M_{m,-1}(r)$ has left-orderable fundamental group if $r \in (-4m,4]$.

\textit{\underline{Case 3}: $|n| \ge 2$.} Proposition \ref{image} and Lemma \ref{lift} imply that $M_{m,n}(r)$ has left-orderable fundamental group if the slope $r$ satisfies the condition $r \in (-4m,0]$. 

If $n \ge 2$, then since $J(2m,2n) \cong J(2n,2m)$, $M_{m,n}(r)$ also has left-orderable fundamental group if $r \in (-4n,0]$. Hence we conclude that $M_{m,n}(r)$ has left-orderable fundamental group $r \in (-\max\{4m,4n\},0]$. 

If $n \le -2$, then since $J(2m,2n) \cong J(2n,2m)$ is the mirror image of $J(-2n,-2m)$, $M_{m,n}(r)$ also has left-orderable fundamental group if $r \in [0,-4n)$. Hence we conclude that $M_{m,n}(r)$ has left-orderable fundamental group if $r \in (-4m,-4n)$.

This completes the proof of Theorem \ref{main}.

\end{document}